\documentclass[oneside]{amsart}

\usepackage[T1]{fontenc}
\usepackage[utf8]{inputenc}
\usepackage{lmodern}

\usepackage{cite}

\title[Area-charge inequality and local rigidity]{Area-charge inequality and local rigidity in charged initial data sets}
\author{Abraão Mendes}
\address{Instituto de Matemática, Universidade Federal de Alagoas, Maceió, AL, Brazil.}
\email{abraao.mendes@im.ufal.br}

\numberwithin{equation}{section}

\newtheorem{theorem}{Theorem}[section]
\newtheorem{proposition}[theorem]{Proposition}

\renewcommand{\div}{\operatorname{div}}
\newcommand{\Ric}{\operatorname{Ric}}
\newcommand{\tr}{\operatorname{tr}}
\newcommand{\p}{\partial}
\renewcommand{\L}{\mathcal{L}}
\newcommand{\A}{\mathcal{A}}
\newcommand{\Q}{\mathcal{Q}}

\usepackage{amssymb}

\begin{document}

\raggedbottom

\begin{abstract}
This paper investigates the geometric consequences of equality in area-charge inequalities for spherical minimal surfaces and, more generally, for marginally outer trapped surfaces (MOTS), within the framework of the Einstein-Maxwell equations. We show that, under appropriate energy and curvature conditions, saturation of the inequality $\A \geq 4\pi(\Q_{\rm E}^2 + \Q_{\rm M}^2)$ imposes a rigid geometric structure in a neighborhood of the surface. In particular, the electric and magnetic fields must be normal to the foliation, and the local geometry is isometric to a Riemannian product. We establish two main rigidity theorems: one in the time-symmetric case and another for initial data sets that are not necessarily time-symmetric. In both cases, equality in the area-charge bound leads to a precise characterization of the intrinsic and extrinsic geometry of the initial data near the critical surface.
\end{abstract}

\maketitle

\section{Introduction}

In his influential 1999 paper, G.~W.~Gibbons~\cite{Gibbons} explores the profound interplay between geometry and gravitation, with particular emphasis on the role of inverse mean curvature flow (IMCF) in the understanding of gravitational entropy. Among the key results discussed is the derivation of an \textit{area-charge inequality}, which asserts that, under natural energy conditions, the area $\A$ of a closed, stable minimal surface enclosing an electric or magnetic charge $\Q$ in a time-symmetric initial data set must satisfy
\begin{align}\label{eq:Gibbons.area-charge}
\A \geq 4\pi \Q^2.
\end{align}
As noted in~\cite{Gibbons}, this inequality also extends to maximal initial data sets that are not necessarily time-symmetric.

Inequality \eqref{eq:Gibbons.area-charge} expresses a fundamental geometric constraint imposed by general relativity: the area of a black hole horizon cannot be arbitrarily small for a given charge. In other words, if a black hole  were to have charge $\Q$ but an area smaller than $4\pi \Q^2$, it would contradict physical expectations.

More recently, S.~Dain, J.~L.~Jaramillo, and M.~Reiris~\cite{DainJaramilloReiris} extended inequality~\eqref{eq:Gibbons.area-charge} to the setting of dynamical black holes without making any symmetry assumptions. They showed that, if $\Sigma$ is an orientable, closed, marginally outer trapped surface satisfying the \emph{spacetime stably outermost condition},\footnote{See Definition~3.2 in~\cite{DainJaramilloReiris}.} in a spacetime that obeys the Einstein equations
\begin{align*}
G + \Lambda h = 8\pi\left(T^{\rm EM} + T^{\rm matter}\right),
\end{align*}
with a non-negative cosmological constant $\Lambda$, and where the non-electromagnetic matter field $T^{\rm matter}$ satisfies the dominant energy condition, then the following area-charge inequality holds:
\begin{align}\label{eq:Dain-Jaramillo-Reiris.area-charge}
\A \geq 4\pi (\Q_{\rm E}^2 + \Q_{\rm M}^2),
\end{align}
where $\A$, $\Q_{\rm E}$, and $\Q_{\rm M}$ denote the area, electric charge, and magnetic charge of $\Sigma$, respectively. Notably, no assumption is made that the matter fields are electrically neutral.

The aim of this paper is to investigate the geometric consequences of equality in~\eqref{eq:Gibbons.area-charge} or~\eqref{eq:Dain-Jaramillo-Reiris.area-charge}, formulated in terms of initial data. More precisely, we show that, under suitable conditions, equality in either~\eqref{eq:Gibbons.area-charge} or~\eqref{eq:Dain-Jaramillo-Reiris.area-charge} implies that the initial data set containing $\Sigma$ exhibits a specific, expected geometric structure in a vicinity of $\Sigma$.

Our first result is the following (see Section~\ref{section:preliminaries} for definitions):

\begin{theorem}\label{thm:main1}
Let $(M^3, g)$ be a Riemannian three-manifold with scalar curvature $R$ satisfying
\begin{align}\label{eq:CDEC.time-symmetric}
\frac{1}{2}R \geq \Lambda + |E|^2 + |B|^2,
\end{align}
where $\Lambda$ is a non-negative constant representing the cosmological constant, and $E$ and $B$ are divergence-free vector fields on $M$ representing the electric and magnetic fields, respectively. 

If $\Sigma$ is an area-minimizing two-sphere embedded in $(M, g)$, then the area, electric charge, and magnetic charge of $\Sigma$ satisfy
\begin{align*}
\A \geq 4\pi (\Q_{\rm E}^2 + \Q_{\rm M}^2).
\end{align*}
Moreover, if equality holds, then there exists a neighborhood $U \cong (-\delta, \delta) \times \Sigma$ of $\Sigma$ in $M$ such that:
\begin{enumerate}
\item The electric and magnetic fields are normal to the foliation; more precisely,
      \begin{align*}
      E = a \nu_t, \quad B = b \nu_t,
      \end{align*}
      for some constants $a$ and $b$, where $\nu_t$ is the unit normal to $\Sigma_t \cong \{t\} \times \Sigma$ along the foliation.
\item $(U, g)$ is isometric to $((-\delta, \delta) \times \Sigma, dt^2 + g_0)$ for some $\delta > 0$, where the induced metric $g_0$ on $\Sigma$ has constant Gaussian curvature
      \begin{align*}
      \kappa_\Sigma = a^2 + b^2.
      \end{align*}
\item The cosmological constant $\Lambda$ equals zero.
\end{enumerate}
\end{theorem}

Our second result is a generalization of Theorem~\ref{thm:main1} to initial data sets that are not necessarily time-symmetric. It reads as follows (see Section~\ref{section:preliminaries} for definitions):

\begin{theorem}\label{thm:main2}
Let $(M^3, g, K, E, B)$ be a three-dimensional initial data set for the Einstein-Maxwell equations satisfying the charged dominant energy condition
\begin{align}\label{eq:CDEC}
\mu + J(v) \ge \Lambda + |E|^2 + |B|^2 - 2\langle E \times B, v \rangle
\end{align}
for every unit vector $v \in T_pM$, every point $p \in M$, and some constant $\Lambda \ge 0$. Assume that $E$ and $B$ are divergence-free and that $K$ is two-convex.

Let $\Sigma$ be a weakly outermost, spherical MOTS in $(M, g, K)$. Then the area, electric charge, and magnetic charge of $\Sigma$ satisfy
\begin{align*}
\A \ge 4\pi(\Q_{\rm E}^2 + \Q_{\rm M}^2).
\end{align*}
Moreover, if equality holds, then there exists an outer neighborhood $U \cong [0,\delta) \times \Sigma$ of $\Sigma$ in $M$ such that:
\begin{enumerate}
\item The electric and magnetic fields are normal to the foliation; more precisely,
\begin{align*}
E = a\nu_t, \quad B = b\nu_t,
\end{align*}
for some constants $a$ and $b$, where $\nu_t$ is the unit normal to $\Sigma_t \cong \{t\} \times \Sigma$ along the foliation.
\item $(U, g)$ is isometric to $([0,\delta) \times \Sigma, dt^2 + g_0)$ for some $\delta > 0$, where the induced metric $g_0$ on $\Sigma$ has constant Gaussian curvature
\begin{align*}
\kappa_\Sigma = a^2 + b^2.
\end{align*}
\item The second fundamental form satisfies $K = f dt^2$ on $U$, where $f \in C^\infty(U)$ depends only on $t \in [0,\delta)$.
\item The energy and momentum densities satisfy
\begin{align*}
\mu = a^2 + b^2, \quad J = 0 \quad \text{on} \quad U.
\end{align*}
\item The cosmological constant $\Lambda$ equals zero.
\end{enumerate}
\end{theorem}

In Section~\ref{section:preliminaries}, we derive inequalities \eqref{eq:CDEC.time-symmetric} and \eqref{eq:CDEC} from the dominant energy condition for the energy-momentum tensor $T^{\rm matter}$.

The paper is organized as follows: In Section~\ref{section:preliminaries}, we present some preliminaries necessary for a proper understanding of this work. In Section~\ref{section:proofs}, we provide the proofs of Theorems~\ref{thm:main1} and~\ref{thm:main2}. Finally, Section~\ref{section:model} offers a model illustrating these results.

\section{Preliminaries}\label{section:preliminaries}

Let $(M^3, g, K)$ be a three-dimensional initial data set in a four-dimensional spacetime $(V^4, h)$; that is, $M$ is a spacelike hypersurface in $(V, h)$ with induced metric~$g$ and second fundamental form $K$, taken with respect to the future-directed timelike unit normal to $M$. Assume that $(V, h)$ satisfies the Einstein equations with cosmological constant $\Lambda$:
\begin{align*}
G + \Lambda h = 8\pi\left(T^{\rm EM} + T^{\rm matter}\right),
\end{align*}
where $G = \Ric_h - \frac{1}{2} R_h h$ is the Einstein tensor of $(V,h)$, $T^{\rm EM}$ is the electromagnetic energy-momentum tensor, and $T^{\rm matter}$ is the energy-momentum tensor associated with non-gravitational and non-electromagnetic matter fields.

The electromagnetic energy-momentum tensor $T^{\rm EM}$ is given by
\begin{equation*}
T_{ab}^{\rm EM}=\frac{1}{4\pi}\Big(F_{ac}{F_b}^{c}-\frac{1}{4}F_{cd}F^{cd}h_{ab}\Big),
\end{equation*}
where $F$ is the electromagnetic 2-form, which is also referred to as the \textit{Faraday tensor}.

Let $u$ be the future-directed timelike unit normal vector field along $M$. As is standard, by the Gauss-Codazzi equations,
\begin{align*}
\mu &:= G(u,u) = \frac{1}{2}(R - |K|^2 + \tau^2), \\
J &:= G(u, \cdot) = \div(K - \tau g),
\end{align*}
where $R$ is the scalar curvature of $(M, g)$ and $\tau = \tr K$ is the mean curvature of $M$ in $(V, h)$ with respect to $u$.

The \textit{electric} and \textit{magnetic vector fields} $E$ and $B$ on $M$ are defined in such a way that
\begin{align*}
E_a &= F_{ab} u^b, \\
B_a &= \frac{1}{2} \epsilon_{abc} F^{bc},
\end{align*}
where $\epsilon_{abc}$ is the induced volume form associated with the metric $g$. Specifically, if~$\hat{\epsilon}$ denotes the volume form of the spacetime metric $h$, then $\epsilon_{abc} = u^d \hat{\epsilon}_{dabc}$. In the main results of this paper, we assume the absence of charged matter, that is, we assume that $\div E = \div B = 0$. 

We refer to $(M,g,K,E,B)$ as initial data for the Einstein-Maxwell equations.

Standard calculations give that
\begin{align*}
T^{\rm EM}(u, u) &= \frac{1}{8\pi}(|E|^2 + |B|^2), \\
T^{\rm EM}(u, v) &= -\frac{1}{4\pi} \langle E \times B, v \rangle,
\end{align*}
for any vector $v$ tangent to $M$, where $(E \times B)_a = \epsilon_{abc}E^bB^c$ defines the cross product of $E$ and $B$, which is known in the literature as the \textit{Poynting vector}.

Now assume that $T^{\rm matter}$ satisfies the \textit{dominant energy condition}:
\begin{equation*}
T^{\rm matter}(X, Y) \ge 0 \text{\, for all future-directed causal vectors } X, Y.
\end{equation*}
Therefore,
\begin{equation*}
G(u, u) + \Lambda h(u, u) = 8\pi\left(T^{\rm EM}(u, u) + T^{\rm matter}(u, u)\right) \ge 8\pi T^{\rm EM}(u, u),
\end{equation*}
and thus
\begin{equation*}
\mu \ge \Lambda + |E|^2 + |B|^2.
\end{equation*}
In this case, if $M$ is maximal (in particular, if $M$ is time-symmetric), then
\begin{equation}\label{eq:2.1}
\frac{1}{2}R \ge \Lambda + |E|^2 + |B|^2.
\end{equation}

More generally, when $M$ is not necessarily maximal, it holds that
\begin{align*}
G(u,u+v)+\Lambda h(u,u+v)\ge 8\pi T^{\rm EM}(u,u+v),
\end{align*}
and so
\begin{align}\label{eq:2.2}
\mu+J(v)\ge \Lambda + |E|^2 + |B|^2 - 2\langle E \times B, v \rangle,
\end{align}
for every unit vector $v$ tangent to $M$.

Inequalities~\eqref{eq:2.1} and~\eqref{eq:2.2} are commonly referred to as the \textit{charged dominant energy condition} and have been considered in numerous situations (see, e.g., \cite{AlaeeKhuriYau,BrydenKhuri,CruzLimaSousa,DainJaramilloReiris,GallowayMendes2025,Gibbons,KhuriWeinsteinYamada,Khuri,WeinsteinYamada}). 

Now let $\Sigma^2$ be a closed embedded surface in $M^3$.

In this paper, we assume that $\Sigma$ and $M$ are orientable; in particular, $\Sigma$ is two-sided. Then we fix a unit normal vector field $\nu$ along $\Sigma$; if $\Sigma$ separates $M$, by convention, we say that $\nu$ points to the \textit{outside} of $\Sigma$. 

In the sequel, we are going to present some important definitions to our purposes.

The \textit{electric} and \textit{magnetic charges} of $\Sigma$ are defined, respectively, by
\begin{align*}
\Q_{\rm E}=\frac{1}{4\pi}\int_\Sigma\langle E,\nu\rangle,\quad
\Q_{\rm M}=\frac{1}{4\pi}\int_\Sigma\langle B,\nu\rangle.
\end{align*}

The \textit{null second fundamental forms} $\chi^+$ and $\chi^-$ of $\Sigma$ in $(M,g,K)$ are defined by
\begin{align*}
\chi^+=K|_\Sigma+A,\quad\chi^-=K|_\Sigma-A,
\end{align*}
where $A$ is the second fundamental form of $\Sigma$ in $(M,g)$ with respect to $\nu$; more precisely,
\begin{align*}
A(X,Y)=g(\nabla_X\nu,Y)\quad\mbox{for}\quad X,Y\in\mathfrak{X}(\Sigma),
\end{align*}
where $\nabla$ is the Levi-Civita connection of $(M,g)$.

The \textit{null expansion scalars} or the \textit{null mean curvatures} $\theta^+$ and $\theta^-$ of $\Sigma$ in $(M,g,K)$ with respect to $\nu$ are defined by
\begin{align*}
\theta^+=\tr\chi^+=\tr_\Sigma K+H^{\Sigma},\quad\theta^-=\tr\chi^-=\tr_\Sigma K-H^{\Sigma},
\end{align*}
where $H^{\Sigma}=\tr A$ is the mean curvature of $\Sigma$ in $(M,g)$ with respect to $\nu$. Observe that $\theta^\pm=\tr\chi^\pm$.

After R.~Penrose, $\Sigma$ is said to be \textit{trapped} if both $\theta^+$ and $\theta^-$ are negative.

Restricting our attention to one side, we say that $\Sigma$ is \textit{outer trapped} if $\theta^+$ is negative and \textit{marginally outer trapped} if $\theta^+$ vanishes. In the latter case, we refer to $\Sigma$ as a \textit{marginally outer trapped surface} or a \textit{MOTS}, for short.

Assume that $\Sigma$ is a MOTS in $(M,g,K)$ with respect to a unit normal $\nu$ that is a boundary in $M$; more precisely, $\nu$ points towards a top-dimensional submanifold $M^+\subset M$ such that $\p M^+=\Sigma$. Then we say that 
$\Sigma$ is \textit{outermost} (resp.\ \textit{weakly outermost}) if there is no closed embedded surface in $M^+$ with $\theta^+\le0$ (resp.\ $\theta^+<0$) that is homologous to and different from $\Sigma$.

We say that $\Sigma$ \textit{minimizes area} in $M$ if $\Sigma$ has the least area in its homology class in $M$; \textit{id est}, $\A(\Sigma)\le \A(\Sigma')$ for every closed embedded surface $\Sigma'$ in $M$ that is homologous to $\Sigma$. Similarly, $\Sigma$ is said to be \textit{outer area-minimizing} if $\Sigma$ minimizes area in $M^+$.

An important notion that we are going to recall now is the notion of stability for MOTS introduced by L.~Andersson, M.~Mars, and W.~Simon~\cite{AnderssonMarsSimon2005,AnderssonMarsSimon2008}.

Let $\Sigma$ be a MOTS in $(M,g,K)$ with respect to $\nu$ and $t\to\Sigma_t$ be a variation of $\Sigma=\Sigma_0$ in $M$ with variation vector field $\frac{\p}{\p t}|_{t=0}=\phi\nu$, for some $\phi\in C^\infty(\Sigma)$. Denote by $\theta^\pm(t)$ the null expansion scalars of $\Sigma_t$ with respect to the unit normal $\nu_t$, where $\nu=\nu_t|_{t=0}$. It is well known that (see~\cite{AnderssonMarsSimon2008}) 
\begin{align*}
\frac{\p\theta^+}{\p t}\Big|_{t=0}=-\Delta\phi+2\langle X,\nabla\phi\rangle+(Q-|X|^2+\div X)\phi,
\end{align*}
where $\Delta$ and $\div$ are the Laplace and divergence operators of $\Sigma$ with respect to the induced metric $\langle\,,\,\rangle$, respectively; $X\in\mathfrak{X}(\Sigma)$ is the vector field that is dual to the 1-form $K(\nu,\cdot)|_\Sigma$, and 
\begin{align*}
Q=\kappa_\Sigma-(\mu+J(\nu))-\frac{1}{2}|\chi^+|^2.
\end{align*}
Here $\kappa_\Sigma$ represents the Gaussian curvature of $\Sigma$.

At this point, it is important to emphasize that, in the general case (i.e., when $\Sigma$ is not necessarily a MOTS), the first variation of $\theta^+$ is given by
\begin{align*}
\frac{\partial \theta^+}{\partial t}\Big|_{t=0} = -\Delta \phi + 2\langle X, \nabla \phi \rangle + \Big(Q - |X|^2 + \div X - \frac{1}{2} \theta^+ + \tau \theta^+ \Big)\phi.
\end{align*}

The operator
\begin{align*}
L\phi=-\Delta\phi+2\langle X,\nabla\phi\rangle+(Q-|X|^2+\div X)\phi,\quad\phi\in C^\infty(\Sigma),
\end{align*}
is referred to as the \textit{MOTS stability operator}. It can be proved that $L$ has a real eigenvalue $\lambda_1$, called the \textit{principal eigenvalue} of $L$, such that $\mbox{Re}\,\lambda\ge\lambda_1$ for any complex eigenvalue $\lambda$. Furthermore, the associated eigenfunction $\phi_1$, $L\phi_1=\lambda_1\phi_1$, is unique up to scale and can be chosen to be everywhere positive.

The principal eigenvalue $\lambda_1(\L)$ of the symmetrized operator $\L = -\Delta + Q$ is characterized by the Rayleigh formula:
\begin{align}\label{eq:2.3}
\lambda_1(\L) =\min_{u \in C^\infty(\Sigma) \setminus \{0\}} \frac{\int_\Sigma (|\nabla u|^2 + Q u^2)}{\int_\Sigma u^2}.
\end{align}
Furthermore, the eigenfunctions of $\L$ associated with $\lambda_1(\L)$ are the only functions that attain the minimum in~\eqref{eq:2.3}. 

It was proved by G.~J.~Galloway and R.~Schoen (see~\cite{Galloway,GallowaySchoen}) through direct estimates, and by L.~Andersson, M.~Mars, and W.~Simon~\cite{AnderssonMarsSimon2008} using a different method, that $\lambda_1(L) \leq \lambda_1(\L)$.

We say that $\Sigma$ is \textit{stable} if $\lambda_1(L)\ge0$; this is equivalent to saying that $L\phi\ge0$ for some positive function $\phi\in C^\infty(\Sigma)$. It is not difficult to see that, if $\Sigma$ is weakly outermost (in particular, if $\Sigma$ is outermost), then $\Sigma$ is stable.

Before concluding this section, let us recall the notion of \textit{2-convexity}. The tensor $K$ is said to be \textit{2-convex} if, at every point, the sum of its two smallest eigenvalues is non-negative. In particular, if $K$ is 2-convex, then $\tr_\Sigma K \ge 0$ along $\Sigma$. This convexity condition has been employed by the author in related contexts~\cite{Mendes2019,EichmairGallowayMendes,Mendes2022,GallowayMendes2024,deAlmeidaMendes} (see also~\cite{LimaSousaBatista}).

\section{Proofs}\label{section:proofs}

This section is devoted to the proofs of the main results of the paper, namely Theorems~\ref{thm:main1} and~\ref{thm:main2}. We begin by proving Theorem~\ref{thm:main2}. The proof of Theorem~\ref{thm:main1} follows a similar structure.

As a first step, we establish an auxiliary \textit{infinitesimal rigidity result}, which plays a crucial role in the argument.

For convenience, we define the \textit{total charge} of~$\Sigma$ as
\begin{align*}
\Q_{\rm T} = \sqrt{\Q_{\rm E}^2 + \Q_{\rm M}^2}.
\end{align*}

\begin{proposition}\label{proposition:infinitesimal}
Let $(M^3, g, K, E, B)$ be a three-dimensional initial data set for the Einstein-Maxwell equations satisfying the charged dominant energy condition
\begin{align*}
\mu + J(v) \ge \Lambda + |E|^2 + |B|^2 - 2\langle E \times B, v \rangle
\end{align*}
for every unit vector $v \in T_pM$, every point $p \in M$, and some constant $\Lambda \ge 0$. 

Let $\Sigma$ be a stable, spherical MOTS in $(M, g, K)$. Then the area and total charge of $\Sigma$ satisfy
\begin{align}\label{eq:proposition:infinitesimal}
\A \ge 4\pi \Q_{\rm T}^2.
\end{align}
Moreover, if equality holds, then the following conditions are satisfied:
\begin{enumerate}
    \item The normal components of the electric and magnetic fields along $\Sigma$ are constant, say $\langle E, \nu \rangle = a$ and $\langle B, \nu \rangle = b$.
    \item $\Sigma$ is a round two-sphere with constant Gaussian curvature $\kappa_\Sigma = a^2 + b^2$.
    \item The constants $\lambda_1(L)$, $\lambda_1(\L)$, and $\Lambda$ equal zero.
\end{enumerate}
\end{proposition}

\begin{proof}
Since $ \Sigma $ is stable and $ \lambda_1(L) \leq \lambda_1(\L) $, we have the following inequality for every $ u \in C^\infty(\Sigma)$:
\begin{align*}
0 \leq \lambda_1(\L) \int_\Sigma u^2 \leq \int_\Sigma(|\nabla u|^2 + Q u^2).
\end{align*}

Taking $ u \equiv 1 $, we obtain
\begin{align}\label{eq:aux3.2}
0 \leq \lambda_1(\L) \A \leq \int_\Sigma Q &= \int_\Sigma\Big(\kappa_\Sigma - (\mu+J(\nu)) - \frac{1}{2} |\chi^+|^2\Big) \\
&\leq 4\pi - \int_\Sigma(\mu+J(\nu)),\nonumber
\end{align}
where we have used the Gauss-Bonnet theorem.

Now observe that
\begin{align}
\mu+J(\nu) &\geq \Lambda+|E|^2+|B|^2-2\langle E\times B,\nu\rangle\label{eq:aux3.3}\\
&\geq |E^\top|^2+|B^\top|^2-2\langle E^\top\times B^\top,\nu\rangle+\langle E,\nu\rangle^2+\langle B,\nu\rangle^2,\nonumber
\end{align}
where $E^\top$ and $B^\top$ are the tangent components to $\Sigma$:  
\begin{align*}
E^\top=E-\langle E,\nu\rangle\nu, \quad B^\top=B-\langle B,\nu\rangle\nu.
\end{align*}

On the other hand, it is not difficult to see that
\begin{align}\label{eq:aux3.4}
|E^\top|^2+|B^\top|^2-2\langle E^\top\times B^\top,\nu\rangle\ge(|E^\top|-|B^\top|)^2\ge0.
\end{align}

Using these estimates, we conclude that
\begin{align*}
0 \leq 4\pi - \int_\Sigma(\langle E, \nu \rangle^2 + \langle B, \nu \rangle^2).
\end{align*}

Applying the Cauchy-Schwarz inequality, we obtain
\begin{align}\label{eq:aux3.5}
(4\pi \Q_{\rm E})^2 = \left( \int_\Sigma \langle E, \nu \rangle \right)^2 \leq \A \int_\Sigma \langle E, \nu \rangle^2.
\end{align}
Similarly,
\begin{align}\label{eq:aux3.6}
(4\pi \Q_{\rm M})^2 \leq \A \int_\Sigma \langle B, \nu \rangle^2.
\end{align}

Therefore,
\begin{align*}
0 \leq \A - 4\pi(\Q_{\rm E}^2 + \Q_{\rm M}^2) = \A - 4\pi \Q_{\rm T}(\Sigma)^2, 
\end{align*}
proving the desired inequality.

If equality in~\eqref{eq:proposition:infinitesimal} holds, then all inequalities above must also be equalities. In particular:
\begin{itemize}
\item Second equality in~\eqref{eq:aux3.5} implies that $\langle E, \nu \rangle$ is constant, say $\langle E, \nu \rangle = a$. Similarly, equality in~\eqref{eq:aux3.6} gives that $\langle B, \nu \rangle = b$ is constant.
\item Equalities in~\eqref{eq:aux3.3} and \eqref{eq:aux3.4} imply \begin{align*}
\Lambda=0,\quad\mu+J(\nu)=\langle E,\nu\rangle^2+\langle B,\nu\rangle^2=a^2+b^2.
\end{align*}

\item Equalities in~\eqref{eq:aux3.2} furnish that $\lambda_1(\L) = 0$, $\chi^+ = 0$, and $u \equiv 1$ is an eigenfunction of $\L$ associated with $\lambda_1(\L)$. Therefore,
\begin{align*}
0 = Q = \kappa_\Sigma - (\mu+J(\nu)),
\end{align*}
and thus
\begin{align*}
\quad \kappa_\Sigma = \mu+J(\nu) = a^2 + b^2.
\end{align*}
\end{itemize}

Finally, since $0 \le \lambda_1(L) \leq \lambda_1(\L) = 0$, we conclude that $\lambda_1(L) = 0$.
\end{proof}

It is worth noting that Proposition~\ref{proposition:infinitesimal} is a quasi-local statement, in the sense that it depends only on the intrinsic and extrinsic geometric data on~$\Sigma$, not on the behavior of the initial data set in a neighborhood of the surface.

\begin{proof}[Proof of Theorem~\ref{thm:main2}]
Since $\Sigma$ is weakly outermost and, in particular, stable, it follows from the infinitesimal rigidity (Proposition~\ref{proposition:infinitesimal}) that
\begin{align*}
\A \ge 4\pi \Q_{\rm T}^2.
\end{align*}
Furthermore, if equality holds, then $\lambda_1(L) = 0$ (and $\Lambda = 0$). Thus, an outer neighborhood $U \cong [0, \delta) \times \Sigma$ of $\Sigma$ in $M$ is foliated by constant null mean curvature surfaces $\Sigma_t \cong \{t\} \times \Sigma$ (see \cite[Lemma~2.3]{Galloway2018}), with $\Sigma_0 = \Sigma$ and
\begin{align*}
    g = \phi^2 dt^2 + g_t \quad \text{on} \quad U,
\end{align*}
where $g_t$ is the induced metric on $\Sigma_t$. 

On $\Sigma_t$, we recall that
\begin{align*}
    \frac{d\theta}{dt} = -\Delta \phi + 2\langle X, \nabla \phi \rangle + \Big(Q - |X|^2 + \div X - \frac{1}{2} \theta^2 + \tau \theta \Big) \phi,
\end{align*}
where $\theta = \theta(t)$ is the null mean curvature of $\Sigma_t$ with respect to $\nu_t = \phi^{-1} \partial_t$.

Dividing both sides of last equation by $\phi$ and integrating over $\Sigma_t$, we obtain
\begin{align}\label{eq:aux3.7}
\theta'\int_{\Sigma_t}\frac{1}{\phi}-\theta\int_{\Sigma_t}\tau&=\int_{\Sigma_t}\Big(\div Y-|Y|^2+Q-\frac{1}{2}\theta^2\Big)\le\int_{\Sigma_t}Q\\
&=\int_{\Sigma_t}\Big(\kappa_{\Sigma_t}-(\mu+J(\nu_t))-\frac{1}{2}|\chi^+_t|^2\Big)\nonumber\\
&\le4\pi-\int_{\Sigma_t}(\mu+J(\nu_t)),\nonumber
\end{align}
where $Y=X-\nabla\ln\phi$.

Using the proof strategy from Proposition~\ref{proposition:infinitesimal}, we have
\begin{align}\label{eq:aux3.8}
\mu+J(\nu_t)\ge|E|^2+|B|^2-2\langle E\times B,\nu_t\rangle\ge\langle E,\nu_t\rangle^2+\langle B,\nu_t\rangle^2.
\end{align}
Thus,
\begin{align}\label{eq:aux3.9}
\theta'\int_{\Sigma_t}\frac{1}{\phi}-\theta\int_{\Sigma_t}\tau&\le4\pi-\int_{\Sigma_t}(\langle E,\nu_t\rangle^2+\langle B,\nu_t\rangle^2)\\
&\le4\pi-\frac{\displaystyle\left(\int_{\Sigma_t}\langle E,\nu_t\rangle\right)^2+\left(\int_{\Sigma_t}\langle B,\nu_t\rangle\right)^2}{\A(t)}\nonumber\\
&=4\pi\left(1-\frac{4\pi \Q_{\rm T}(t)^2}{\A(t)}\right),\nonumber
\end{align}
where $\A(t)$ and $\Q_{\rm T}(t)$ are the area and total charge of $\Sigma_t$, respectively.

Because we are assuming $4\pi \Q_{\rm T}(0)^2 = \A(0)$ and $\div E = \div B = 0$ (implying $\Q_{\rm T}(t) = \Q_{\rm T}(0)$), we find
\begin{align}\label{eq:aux3.10}
\theta'\left(\frac{\A(t)}{4\pi}\int_{\Sigma_t}\frac{1}{\phi}\right)-\theta\left(\frac{\A(t)}{4\pi}\int_{\Sigma_t}\tau\right)&\le \A(t)-4\pi \Q_{\rm T}(t)^2\\
&=\A(t)-\A(0)\nonumber\\
&=\int_0^t\left(\int_{\Sigma_s}H^{\Sigma_s}\phi\right)ds,\nonumber
\end{align}
where we have used the fundamental theorem of calculus along with the first\linebreak variation of area formula.

Since, by hypothesis, $K$ is 2-convex, it follows that
\begin{align}\label{eq:aux3.11}
    H^{\Sigma_s} \leq \tr_{\Sigma_s} K + H^{\Sigma_s} = \theta(s).
\end{align}
Therefore,
\begin{align*}
    \theta'(t) \left(\frac{\A(t)}{4\pi} \int_{\Sigma_t} \frac{1}{\phi}\right) - \theta(t) \left(\frac{\A(t)}{4\pi} \int_{\Sigma_t} \tau\right) \leq \int_0^t \theta(s) \left( \int_{\Sigma_s} \phi \right) ds.
\end{align*}

Using Lemma~3.2 in \cite{Mendes2019}, we conclude that $\theta(t) \leq 0$. Because $\Sigma$ is weakly outermost, this implies $\theta(t) = 0$, forcing all inequalities above to be equalities.

Thus:
\begin{itemize}
\item Equalities in \eqref{eq:aux3.11} give that $\tr_{\Sigma_t}K=H^{\Sigma_t}=0$ along $\Sigma_t$. In particular,
\begin{align*}
\theta^-(t)=\tr_{\Sigma_t}K-H^{\Sigma_t}=0
\end{align*}
for every $t\in[0,\delta)$.
\item Equalities in \eqref{eq:aux3.10} imply that all surfaces $\Sigma_t$ have the same area as $\Sigma$: $\A(t)=\A(0)$.
\item Equalities in \eqref{eq:aux3.8} hold, that is,
\begin{align}\label{eq:aux3.12}
\mu+J(\nu_t)=|E|^2+|B|^2-2\langle E\times B,\nu_t\rangle=\langle E,\nu_t\rangle^2+\langle B,\nu_t\rangle^2
\end{align}
along $\Sigma_t$ for every $t\in[0,\delta)$.
\item Finally, equalities in \eqref{eq:aux3.7} imply $Y=X-\nabla\ln\phi=0$ and $\chi_t^+=0$ along $\Sigma_t$.
\end{itemize}

Now, taking the first variation of $\theta^-(t)=0$, with $\phi^-=-\phi$ instead of $\phi$, we obtain
\begin{align}\label{eq:aux3.13}
0=\frac{d\theta^-}{dt}=-\Delta\phi^-+2\langle X^-,\nabla\phi^-\rangle+(Q^--|X^-|^2+\div X^-)\phi^-,
\end{align}
where $X^-=(K(-\nu_t,\cdot)|_{\Sigma_t})^\sharp=-X=-\nabla\ln\phi$, and
\begin{align*}
Q^-=\kappa_{\Sigma_t}-(\mu-J(\nu_t))-\frac{1}{2}|\chi^-_t|^2.
\end{align*}

Thus, dividing both sides of \eqref{eq:aux3.13} by $\phi^-=-\phi$ and integrating over $\Sigma_t$, we get
\begin{align}\label{eq:aux3.14}
0=\int_{\Sigma_t}(\div Y^--|Y^-|^2+Q^-)\le\int_{\Sigma_t}Q^-\le4\pi-\int_{\Sigma_t}(\mu-J(\nu_t)),
\end{align}
where $Y^-=X^--\nabla\ln\phi=-2\nabla\ln\phi$. Above we have used the Gauss-Bonnet theorem.

Observe that
\begin{align}\label{eq:aux3.15}
\mu-J(\nu_t)\ge|E|^2+|B|^2+2\langle E\times B,\nu_t\rangle\ge\langle E,\nu_t\rangle^2+\langle B,\nu_t\rangle^2.
\end{align}
Therefore,
\begin{align*}
0&\le4\pi-\int_{\Sigma_t}(\mu-J(\nu_t))\le4\pi-\int_{\Sigma_t}(\langle E,\nu_t\rangle^2+\langle B,\nu_t\rangle^2)\\
&\le4\pi\left(1-\frac{4\pi \Q_{\rm T}(t)^2}{\A(t)}\right)=4\pi\left(1-\frac{\A(0)}{\A(t)}\right)=0,
\end{align*}
thus all inequalities above must be equalities. 

Then:
\begin{itemize}
\item From \eqref{eq:aux3.12} and equalities in \eqref{eq:aux3.15}, we have
\begin{align*}
|E|^2+|B|^2-2\langle E\times B,\nu_t\rangle&=\langle E,\nu_t\rangle^2+\langle B,\nu_t\rangle^2\\
&=|E|^2+|B|^2+2\langle E\times B,\nu_t\rangle.
\end{align*}
Therefore, $\langle E\times B,\nu_t\rangle=0$ and $|E|^2+|B|^2=\langle E,\nu_t\rangle^2+\langle B,\nu_t\rangle^2$. Thus, $E$ and $B$ are parallel to $\nu_t$, say $E=a\nu_t$ and $B=b\nu_t$. Furthermore, from the second equality in \eqref{eq:aux3.9}, we obtain that $a=a(t)$ and $b=b(t)$ are constant on $\Sigma_t$.
\item Equalities in \eqref{eq:aux3.14} imply $Y^-=-2\nabla\ln\phi=0$ and $\chi^-_t=0$ along $\Sigma_t$. Therefore, $\phi=\phi(t)$ is constant on $\Sigma_t$ for each $t\in[0,\delta)$. In this case, after a change of variable if necessary, we may assume that $\phi\equiv1$. Moreover, $\chi^+_t=K|_{\Sigma_t}+A_t=0$ and $\chi^-_t=K|_{\Sigma_t}-A_t=0$ imply that $K|_{\Sigma_t}=0$ and $\Sigma_t$ is totally geodesic in $(M,g)$. This gives that
\begin{align*}
g=dt^2+g_0\quad\mbox{on}\quad U\cong[0,\delta)\times\Sigma,
\end{align*}
where $g_0$ is the induced metric on $\Sigma$.
\item Because $\div E=\div B=0$, we can see that $a$ e $b$ are constant.
\item Looking at \eqref{eq:aux3.12} and equalities in \eqref{eq:aux3.15} again, we get
\begin{align*}
\mu+J(\nu_t)=a^2+b^2=\mu-J(\nu_t).
\end{align*}
Therefore,
\begin{align*}
\mu=a^2+b^2,\quad J(\nu_t)=0.
\end{align*}
\item It follows from \eqref{eq:aux3.13} that
\begin{align*}
0=Q^-=\kappa_{\Sigma_t}-\mu\quad\therefore\quad \kappa_{\Sigma_t}=\mu=a^2+b^2.
\end{align*}  
\item Given a unit vector $v$ tangent to $M$, we are assuming that
\begin{align*}
\mu+J(v)\ge|E|^2+|B|^2-2\langle E\times B,v\rangle=a^2+b^2=\mu.
\end{align*}
Therefore, $J(v)\ge0$ for every $v$, that is, $J=0$.
\item Finally, $K|_{\Sigma_t}=0$, $K(\nu_t,\cdot)|_{\Sigma_t}=X^\flat=0$, and $J=\div(K-\tau g)=0$ give that $K=fdt^2$ on $U\cong[0,\delta)\times\Sigma$, where $f$ depends only on $t\in[0,\delta)$.
\end{itemize}

This concludes the proof of Theorem~\ref{thm:main2}.
\end{proof}

We now proceed with the proof of Theorem~\ref{thm:main1}. To this end, we first present the following auxiliary result:

\begin{proposition}\label{proposition:infinitesimal.Riemannian}
Let $(M^3, g)$ be a three-dimensional Riemannian manifold whose scalar curvature $R$ satisfies
\begin{align*}
\frac{1}{2}R \geq \Lambda + |E|^2 + |B|^2,
\end{align*}
where $\Lambda$ is a non-negative constant, and $E$ and $B$ are vector fields on $M$.

If $\Sigma$ is a stable, minimal two-sphere embedded in $(M, g)$, then the area and total charge of $\Sigma$ satisfy
\begin{align}\label{eq:Gibbons.area-charge2}
\A \ge 4\pi \Q_{\rm T}^2.
\end{align}
Moreover, if equality holds, then the following conditions are satisfied:
\begin{enumerate}
    \item The electric and magnetic fields are parallel to $\nu$; more precisely,
\begin{align*}
E = a\nu, \quad B = b\nu,
\end{align*}
for some constants $a$ and $b$.
    \item $\Sigma$ is a round two-sphere with constant Gaussian curvature $\kappa_\Sigma = a^2 + b^2$.
    \item $\Sigma$ is totally geodesic, $\Lambda=0$, and $R=2(a^2+b^2)$ on $\Sigma$.
\end{enumerate}
\end{proposition}

It is worth noting that inequality~\eqref{eq:Gibbons.area-charge2} was originally derived by Gibbons~\cite{Gibbons} (see also~\cite[Theorem~4.4]{DainJaramilloReiris}). Our contribution lies in establishing the infinitesimal rigidity statement.

\begin{proof}
Since $\Sigma$ is a stable minimal surface, the stability inequality says that
\begin{align*}
\frac{1}{2}\int_\Sigma(R + |A|^2)u^2 \le \int_\Sigma |\nabla u|^2 + \int_\Sigma \kappa_\Sigma u^2
\end{align*}
for every $ u \in C^\infty(\Sigma) $. Taking $ u \equiv 1 $, we obtain
\begin{align}\label{eq:aux3.17}
\frac{1}{2}\int_\Sigma R \le 4\pi,
\end{align}
where we have used the Gauss-Bonnet theorem.

Next, using the estimate
\begin{align}\label{eq:aux3.18}
\frac{1}{2}R \ge \Lambda + |E|^2 + |B|^2 \ge \langle E, \nu \rangle^2 + \langle B, \nu \rangle^2,
\end{align}
we conclude that
\begin{align*}
\int_\Sigma( \langle E, \nu \rangle^2 + \langle B, \nu \rangle^2 ) \le 4\pi.
\end{align*}

Finally, applying the Cauchy-Schwarz inequality yields
\begin{align}\label{eq:aux3.19}
(4\pi \Q_{\rm T})^2 
= \left( \int_\Sigma \langle E, \nu \rangle \right)^2 
+ \left( \int_\Sigma \langle B, \nu \rangle \right)^2 
\le \A \int_\Sigma( \langle E, \nu \rangle^2 + \langle B, \nu \rangle^2) 
\le 4\pi \A,
\end{align}
which proves inequality~\eqref{eq:Gibbons.area-charge2}.

Now suppose that equality holds in~\eqref{eq:Gibbons.area-charge2}. Then equality must also hold in each of the steps above.

Equality in~\eqref{eq:aux3.17} implies that $ \Sigma $ is totally geodesic and that $ u_0 \equiv 1 $ is a Jacobi function on $\Sigma$:
\begin{align*}
\Delta u_0 + \frac{1}{2}(R - 2\kappa_\Sigma + |A|^2)u_0 = 0 \quad \text{on} \quad \Sigma.
\end{align*}
Therefore, $ R = 2\kappa_\Sigma $.

Equality in~\eqref{eq:aux3.18} implies $ \Lambda = 0 $, and that $ E $ and $ B $ are parallel to $ \nu $ along $\Sigma$, i.e.\ $ E = a\nu $ and $ B = b\nu $ for some functions $ a, b $.

Finally, second equality in~\eqref{eq:aux3.19} implies that $ \langle E, \nu \rangle $ and $ \langle B, \nu \rangle $ are constant, hence $ a $ and $ b $ are constant functions.
\end{proof}

\begin{proof}[Proof of Theorem~\ref{thm:main1}]

Since $\Sigma$ is area-minimizing (in particular, stable minimal), it follows from Proposition~\ref{proposition:infinitesimal.Riemannian} that $\A \ge 4\pi \Q_{\rm T}^2$. Furthermore, if equality holds, then the Jacobi operator of $\Sigma$ reduces to $-\Delta$.

Therefore, as in the proof of Theorem~\ref{thm:main2}, by a classical argument in the literature (see, e.g., \cite{AnderssonCaiGalloway,BrayBrendleNeves,MicallefMoraru,Nunes}), a neighborhood $U \cong (-\delta,\delta) \times \Sigma$ of $\Sigma$ in $M$ can be foliated by constant mean curvature surfaces $\Sigma_t \cong \{t\} \times \Sigma$, with $\Sigma_0 = \Sigma$ and
\begin{align*}
g = \phi^2 dt^2 + g_t \quad \text{on} \quad U.
\end{align*}

The first variation of $H(t) := H^{\Sigma_t}$ gives
\begin{align*}
H' = -\Delta \phi - \frac{1}{2}(R - 2\kappa_{\Sigma_t} + |A_t|^2 + H^2)\phi.
\end{align*}
Thus,
\begin{align*}
H' \int_{\Sigma_t} \frac{1}{\phi} &= -\int_{\Sigma_t} \frac{\Delta \phi}{\phi} - \frac{1}{2} \int_{\Sigma_t} (R + |A_t|^2 + H^2) + \int_{\Sigma_t} \kappa_{\Sigma_t} \\
&\le -\int_{\Sigma_t} \frac{|\nabla \phi|^2}{\phi^2} - \frac{1}{2} \int_{\Sigma_t} R + 4\pi \\
&\le -\frac{1}{2} \int_{\Sigma_t} R + 4\pi.
\end{align*}

Using the estimates
\begin{align*}
\frac{1}{2} R \ge |E|^2 + |B|^2 \ge \langle E, \nu_t \rangle^2 + \langle B, \nu_t \rangle^2,
\end{align*}
and applying the Cauchy-Schwarz inequality, we obtain
\begin{align*}
H' \int_{\Sigma_t} \frac{1}{\phi} \le 4\pi \left(1 - \frac{4\pi \Q_{\rm T}(t)^2}{\A(t)}\right).
\end{align*}

On the other hand, $4\pi \Q_{\rm T}(t)^2 = 4\pi \Q_{\rm T}(0)^2 = \A(0)$, since $\div E = \div B = 0$. Therefore,
\begin{align*}
H'(t) \int_{\Sigma_t} \frac{1}{\phi} \le \frac{4\pi}{\A(t)} (\A(t) - \A(0)) = \frac{4\pi}{\A(t)} \int_0^t H(s) \left( \int_{\Sigma_s} \phi \right) ds,
\end{align*}
that is,
\begin{align*}
H'(t) \eta(t) \le \int_0^t H(s) \xi(s) ds, \quad \eta(t) := \frac{\A(t)}{4\pi} \int_{\Sigma_t} \frac{1}{\phi}, \quad \xi(t) := \int_{\Sigma_t} \phi,
\end{align*}
where we have used the fundamental theorem of calculus together with the first variation of area formula. This holds for every $t \in (-\delta, \delta)$.

It follows directly from Lemma~3.2 in~\cite{Mendes2019} that $H(t) \le 0$ for every $t \in [0, \delta)$. Similarly, by applying the same strategy as in the proof of Lemma~3.2 in~\cite{Mendes2019} for $\rho(t) = 0$, it is not difficult to show that $H(t) \ge 0$ for every $t \in (-\delta, 0]$. Therefore,
\begin{align}\label{eq:aux3.20}
\A'(t)=\int_{\Sigma_t}H(t)\phi
\left\{
\begin{array}{ll}
\le 0&\mbox{for}\quad t\in[0,\delta),\\
\ge 0&\mbox{for}\quad t\in(-\delta,0].
\end{array}
\right. 
\end{align}
In any case, $\A(t) \le \A(0)$ for every $t \in (-\delta, \delta)$. This implies that $\A(t) = \A(0)$ for all $t \in (-\delta, \delta)$, since $\Sigma_0 = \Sigma$ is area-minimizing. Using this in~\eqref{eq:aux3.20}, we obtain that $H(t) = 0$ for every $t \in (-\delta, \delta)$. Therefore, all inequalities above must be equalities.

Thus, each $\Sigma_t$ is an area-minimizing surface satisfying $\A(t) = 4\pi \Q_{\rm T}(t)^2$. Then, by Proposition~\ref{proposition:infinitesimal.Riemannian},
\begin{itemize}
\item $E = a \nu_t$ and $B = b \nu_t$, where $a = a(t)$ and $b = b(t)$ are constant on $\Sigma_t$;
\item Each $\Sigma_t$ is a totally geodesic round two-sphere with constant Gaussian curvature $\kappa_{\Sigma_t} = a^2 + b^2$;
\item $R = 2(a^2 + b^2)$ on $\Sigma_t$ for each $t \in (-\delta, \delta)$.
\end{itemize}

Finally, since $\div E = \div B = 0$, we conclude that $a$ and $b$ are constant functions. Standard calculations guarantee Theorem~\ref{thm:main1}.
\end{proof}

\section{The model}\label{section:model}

Let $q > 0$ and consider the dyonic Bertotti-Robinson spacetime $(V^4, h)$ defined by  
\begin{align*}
V^4 = \mathbb{R} \times \mathbb{R} \times S^2, \quad h = q^2(-\cosh^2r\, dt^2 + dr^2 + d\theta^2 + \sin^2\theta\, d\phi^2).
\end{align*}

Note that $(V, h)$ is the direct product of a two-dimensional anti-de Sitter space of curvature $-1/q^2$ and a round two-sphere of curvature $1/q^2$. Consequently, one can verify that
\begin{align*}
\Ric_h = \frac{1}{q^2} \operatorname{diag}(-h_{tt}, -h_{rr}, h_{\theta\theta}, h_{\phi\phi}).
\end{align*}
In particular, the scalar curvature of $(V, h)$ vanishes.

Now let $q_e$ and $q_m$ be constants and define the Faraday tensor $F$ by  
\begin{align*}
F = -q_e \cosh r\, dt \wedge dr + q_m \sin\theta\, d\theta \wedge d\phi.
\end{align*}
A direct computation shows that the associated electromagnetic energy-momentum tensor $T^{\mathrm{EM}}$ takes the form  
\begin{align*}
T^{\mathrm{EM}} = \frac{1}{8\pi} \frac{q_e^2 + q_m^2}{q^4} \operatorname{diag}(-h_{tt}, -h_{rr}, h_{\theta\theta}, h_{\phi\phi}).
\end{align*}
Therefore, by choosing $q_e$ and $q_m$ such that $q^2 = q_e^2 + q_m^2$, the spacetime $(V, h)$ satisfies the Einstein equations with zero cosmological constant:
\begin{align*}
\Ric_h = 8\pi T^{\mathrm{EM}}.
\end{align*}

Observe that each $t$-slice $M = \{t\} \times \mathbb{R} \times S^2$ is time-symmetric and isometric to the Riemannian product of a line with a round two-sphere of Gaussian curvature $1/q^2$. Furthermore, the electric and magnetic vector fields on $M$ are given by  
\begin{align*}
E = \frac{q_e}{q^2} \nu, \quad B = \frac{q_m}{q^2} \nu, \quad \nu := \frac{1}{q} \partial_r.
\end{align*}

Finally, consider the 2-sphere $\Sigma = \{t\} \times \{r\} \times S^2$. The electric charge enclosed by $\Sigma$ is  
\begin{align*}
\mathcal{Q}_{\mathrm{E}} = \frac{1}{4\pi}\int_\Sigma \langle E, \nu \rangle = \frac{1}{4\pi} \frac{q_e}{q^2} \mathcal{A} = q_e.
\end{align*}
Similarly, the magnetic charge is  
\begin{align*}
\mathcal{Q}_{\mathrm{M}} = q_m.
\end{align*}

Clearly, $\Sigma$ and $M$ satisfy all the assumptions of Theorems~\ref{thm:main1} and~\ref{thm:main2} with $\A=4\pi\Q_{\rm T}^2$.

For a detailed discussion of the Bertotti-Robinson spacetime with $q_m=0$, as well as other notable spacetimes in dimension $D \ge 4$ with vanishing magnetic field, see~\cite{CardosoDiasLemos}.

\section*{Acknowledgments}

The author sincerely thanks Greg Galloway for his kind interest in this work. He also gratefully acknowledges partial support from the Conselho Nacional de Desenvolvimento Científico e Tecnológico (CNPq, Grant 309867/2023-1) and the Coordenação de Aperfeiçoamento de Pessoal de Nível Superior (CAPES/MATH-AMSUD 88887.985521/2024-00), Brazil.

%\section*{Conflict of interest}

%We know of no conflicts of interest associated with this publication, and there has been no significant financial support for this work that could have influenced its outcome.

%\section*{Data availability}

%No data were generated or analyzed in this study.

\bibliographystyle{plain}
\bibliography{bibliography.bib}

\begin{thebibliography}{10}

\bibitem{AlaeeKhuriYau}
Aghil Alaee, Marcus Khuri, and Shing-Tung Yau.
\newblock Geometric inequalities for quasi-local masses.
\newblock {\em Commun. Math. Phys.}, 378(1):467--505, 2020.

\bibitem{AnderssonCaiGalloway}
Lars Andersson, Mingliang Cai, and Gregory~J. Galloway.
\newblock Rigidity and positivity of mass for asymptotically hyperbolic
  manifolds.
\newblock {\em Ann. Henri Poincar{\'e}}, 9(1):1--33, 2008.

\bibitem{AnderssonMarsSimon2005}
Lars Andersson, Marc Mars, and Walter Simon.
\newblock Local existence of dynamical and trapping horizons.
\newblock {\em Phys. Rev. Lett.}, 95:111102, Sep 2005.

\bibitem{AnderssonMarsSimon2008}
Lars Andersson, Marc Mars, and Walter Simon.
\newblock Stability of marginally outer trapped surfaces and existence of
  marginally outer trapped tubes.
\newblock {\em Adv. Theor. Math. Phys.}, 12(4):853--888, 2008.

\bibitem{BrayBrendleNeves}
Hubert Bray, Simon Brendle, and Andre Neves.
\newblock Rigidity of area-minimizing two-spheres in three-manifolds.
\newblock {\em Commun. Anal. Geom.}, 18(4):821--830, 2010.

\bibitem{BrydenKhuri}
Edward~T. Bryden and Marcus~A. Khuri.
\newblock The area-angular momentum-charge inequality for black holes with
  positive cosmological constant.
\newblock {\em Classical Quantum Gravity}, 34(12):24, 2017.
\newblock Id/No 125017.

\bibitem{CardosoDiasLemos}
Vitor Cardoso, \'Oscar J.~C. Dias, and Jos\'e P.~S. Lemos.
\newblock {N}ariai, {B}ertotti-{R}obinson, and anti-{N}ariai solutions in
  higher dimensions.
\newblock {\em Phys. Rev. D}, 70:024002, Jul 2004.

\bibitem{CruzLimaSousa}
Tiarlos Cruz, Vanderson Lima, and Alexandre de~Sousa.
\newblock Min-max minimal surfaces, horizons and electrostatic systems.
\newblock {\em J. Differ. Geom.}, 128(2):583--637, 2024.

\bibitem{DainJaramilloReiris}
Sergio Dain, Jos{\'e}~Luis Jaramillo, and Mart{\'{\i}}n Reiris.
\newblock Area-charge inequality for black holes.
\newblock {\em Classical Quantum Gravity}, 29(3):15, 2012.
\newblock Id/No 035013.

\bibitem{deAlmeidaMendes}
Deivid de~Almeida and Abra{\~a}o Mendes.
\newblock Rigidity results for free boundary hypersurfaces in initial data sets
  with boundary.
\newblock Preprint, {arXiv}:2502.09433 [math.{DG}] (2025), 2025.

\bibitem{EichmairGallowayMendes}
Michael Eichmair, Gregory~J. Galloway, and Abra{\~a}o Mendes.
\newblock Initial data rigidity results.
\newblock {\em Commun. Math. Phys.}, 386(1):253--268, 2021.

\bibitem{Galloway}
Gregory~J. Galloway.
\newblock Stability and rigidity of extremal surfaces in {Riemannian} geometry
  and general relativity.
\newblock In {\em Surveys in geometric analysis and relativity. Dedicated to
  Richard Schoen in honor of his 60th birthday}, pages 221--239. Somerville,
  MA: International Press; Beijing: Higher Education Press, 2011.

\bibitem{Galloway2018}
Gregory~J. Galloway.
\newblock Rigidity of outermost {MOTS}: the initial data version.
\newblock {\em Gen. Relativ. Gravitation}, 50(3):7, 2018.
\newblock Id/No 32.

\bibitem{GallowayMendes2024}
Gregory~J. Galloway and Abra{\~a}o Mendes.
\newblock Some rigidity results for compact initial data sets.
\newblock {\em Trans. Am. Math. Soc.}, 377(3):1989--2007, 2024.

\bibitem{GallowayMendes2025}
Gregory~J. Galloway and Abra{\~a}o Mendes.
\newblock Some rigidity results for charged initial data sets.
\newblock {\em Nonlinear Anal., Theory Methods Appl., Ser. A, Theory Methods},
  256:9, 2025.
\newblock Id/No 113780.

\bibitem{GallowaySchoen}
Gregory~J. Galloway and Richard Schoen.
\newblock A generalization of {Hawking}'s black hole topology theorem to higher
  dimensions.
\newblock {\em Commun. Math. Phys.}, 266(2):571--576, 2006.

\bibitem{Gibbons}
G.~W. Gibbons.
\newblock Some comments on gravitational entropy and the inverse mean curvature
  flow.
\newblock {\em Classical Quantum Gravity}, 16(6):1677--1687, 1999.

\bibitem{KhuriWeinsteinYamada}
Marcus Khuri, Gilbert Weinstein, and Sumio Yamada.
\newblock Proof of the {Riemannian} {Penrose} inequality with charge for
  multiple black holes.
\newblock {\em J. Differ. Geom.}, 106(3):451--498, 2017.

\bibitem{Khuri}
Marcus~A. Khuri.
\newblock Inequalities between size and charge for bodies and the existence of
  black holes due to concentration of charge.
\newblock {\em J. Math. Phys.}, 56(11):112503, 9, 2015.

\bibitem{LimaSousaBatista}
A.~B. Lima, P.~A. Sousa, and R.~M. Batista.
\newblock Rigidity of marginally outer trapped surfaces in charged initial data
  sets.
\newblock {\em Lett. Math. Phys.}, 115(2):15, 2025.
\newblock Id/No 41.

\bibitem{Mendes2019}
Abra{\~a}o Mendes.
\newblock Rigidity of marginally outer trapped (hyper)surfaces with negative
  {{\(\sigma \)}}-constant.
\newblock {\em Trans. Am. Math. Soc.}, 372(8):5851--5868, 2019.

\bibitem{Mendes2022}
Abra{\~a}o Mendes.
\newblock Rigidity of free boundary {MOTS}.
\newblock {\em Nonlinear Anal., Theory Methods Appl., Ser. A, Theory Methods},
  220:15, 2022.
\newblock Id/No 112841.

\bibitem{MicallefMoraru}
Mario Micallef and Vlad Moraru.
\newblock Splitting of 3-manifolds and rigidity of area-minimising surfaces.
\newblock {\em Proc. Am. Math. Soc.}, 143(7):2865--2872, 2015.

\bibitem{Nunes}
Ivaldo Nunes.
\newblock Rigidity of area-minimizing hyperbolic surfaces in three-manifolds.
\newblock {\em J. Geom. Anal.}, 23(3):1290--1302, 2013.

\bibitem{WeinsteinYamada}
Gilbert Weinstein and Sumio Yamada.
\newblock On a {Penrose} inequality with charge.
\newblock {\em Commun. Math. Phys.}, 257(3):703--723, 2005.

\end{thebibliography}

\end{document}